\documentclass[a4paper,reqno]{amsart}
\usepackage{amsmath,amssymb,mathrsfs,latexsym,tikz,cite,hyperref}
\usepackage{mfabacus} 
\usepackage{mathabx}

\allowdisplaybreaks 

\title{Core blocks for Hecke algebras of type $B$ and sign sequences}
\author[S.~Lyle]{Sin\'ead Lyle}
\address{School of Mathematics, University of East Anglia, Norwich NR4 7TJ, UK.}
\email{s.lyle@uea.ac.uk}
\subjclass[2020]{20C08, 20C30, 05E10}
\keywords{Ariki-Koike algebras, abacuses}

\numberwithin{equation}{section}
\numberwithin{figure}{section}
\newtheorem{lemma}{Lemma}[section]
\newtheorem{theorem}[lemma]{Theorem}
\newtheorem{proposition}[lemma]{Proposition}
\newtheorem{corollary}[lemma]{Corollary}

\theoremstyle{remark}

\newtheorem*{ex}{Example}

\newcommand{\la}{\lambda}
\newcommand{\La}{\Lambda}
\newcommand{\bla}{\boldsymbol \la}
\newcommand{\bmu}{\boldsymbol \mu}
\newcommand{\bnu}{\boldsymbol \nu}
\newcommand{\bsig}{\boldsymbol \sigma}
\newcommand{\btau}{\boldsymbol \tau}
\newcommand{\bb}{{\mathbf b}}
\newcommand{\circo}{~\raisebox{1pt}{\tikz \draw[line width=0.6pt] circle(1.4pt);}~}
\newcommand{\bz}{\bb^{\!\!\!\!\circo}\!\!}
\newcommand{\bzi}{b^{\!\!\!\!\circo}_i\!}

\newcommand{\re}{\mathbf r}

\newcommand{\de}{\delta}
\newcommand{\mc}{\mathfrak{a}}
\newcommand{\ms}{\mathfrak{s}}
\newcommand{\rr}{\texttt{r}}
\newcommand{\ab}{\texttt{a}}

\newcommand{\Z}{\mathbb{Z}}
\newcommand{\N}{\mathbb{N}}
\newcommand{\F}{\mathcal{F}}
\newcommand{\h}{\mathcal{H}}
\newcommand{\St}{\mathcal{S}}
\newcommand{\R}{\mathcal{R}}

\DeclareMathOperator{\res}{res}
\DeclareMathOperator{\Res}{Res}

\DeclareMathOperator{\wt}{wt}
\DeclareMathOperator{\Hom}{Hom}
\DeclareMathOperator{\rad}{rad}

\newcommand{\emp}{\varnothing}

\newcommand{\Cr}[1]{s_{#1}}

\newcommand{\ra}{\raisebox{0.1cm}{
\begin{tikzpicture}
\draw [<-] (0,0) to [out =60, in = 120] (.4,0);
\end{tikzpicture}}}
\newcommand{\rra}{\raisebox{0.1cm}{
\begin{tikzpicture}
\draw [<-] (0,0) to [out =60, in = 120] (.4,0);
\draw (0.1,0) to [out =60, in = 120] (.3,0);
\end{tikzpicture}}}

\begin{document}
\begin{abstract}
We consider the core blocks corresponding to the Hecke algebras of type $B$ over a field of arbitrary characteristic. To each core block, we associate two non-negative integers which determine the indexing of the Specht modules and simple modules in the block, the weight of the block, the multicharge of the algebra (up to a shift) and the block decomposition matrix.
\end{abstract}
\maketitle

\section{Introduction}

Let $\h=\h_{r,n}(q,{\bf Q})$ denote an Ariki-Koike algebra over a field $\mathbb{F}$ of characteristic $p$ with quantum characteristic $e \in \{2,3,\ldots\} \cup \{\infty\}$. Each of these algebras decomposes into a direct sum of indecomposable two-sided ideals, its blocks, and so in order to try to understand the algebras one may study certain types of block in the hope that they are more manageable.
 
There is an important class of $\h$-modules which are indexed by the $r$-multipartitions of $n$ and which are known as Specht modules. The composition factors of each Specht module all lie in the same block and so we may think of partitioning the Specht modules into blocks. One may consider the block decomposition matrix, which records the composition factors of the Specht modules belonging to a given block. 
In general, computing block decomposition matrices is hard. There exist recursive algorithms which compute the transition coefficients for a highest weight module of the Fock space representation of $\mathcal{U}_q(\widehat{\mathfrak{sl}}_e)$; by Ariki's theorem~\cite{Ariki}, these coefficients coincide with the decomposition numbers for the Ariki-Koike algebras when $p=0$. When $p>0$, the transition coefficients give a lower bound for the decomposition numbers, and may be considered as a first approximation to them. It is reasonable to ask when this approximation is precise, or when the block decomposition matrix is independent of the characteristic of the field. 

In this paper, we look at the core blocks for the Hecke algebras of type $B$, that is, the Ariki-Koike algebras $\h=\h_{2,n}(q,{\bf Q})$. These core blocks, which were initially introduced and studied by Fayers~\cite{Fayers:Cores}, are blocks in which none of the bipartitions indexing the Specht modules have any removable $e$-rim hooks. If $e=\infty$ or $e>n$, all blocks are core blocks. 
To each core block $B$, we associate two non-negative integers $n_B$ and $p_B$. We show that these two integers determine the indexing of the Specht modules and simple modules in the block, the weight of the block, the multicharge of $\h$ (up to a shift) and the block decomposition matrix. We present a result of George Witty~\cite{WittyThesis} that shows that in some situations they also determine the homomorphism space between pairs of Specht modules in the block. All these results are independent of the characteristic of the field. 

When $e=\infty$ or $e>n$, decomposition numbers for the Hecke algebras of type $B$, or the corresponding cyclotomic $q$-Schur algbras, are not new. When $e=\infty$ and the parameters of the multicharge are weakly increasing, Leclerc and Miyachi~\cite[Theorem~3]{LM:FockI} give a closed formula for the canonical basis elements of the irreducible highest weight representation $V(\Lambda)$; by Ariki's Theorem~\cite{Ariki}, this gives the graded decomposition numbers for the Hecke algebras when $e=\infty$ and $p=0$. 
When $e=\infty$ and $p=0$, Brundan and Stroppel~\cite[Section~9]{BS:KhovIII} apply the theory they have developed earlier in their paper to construct a basis for the radical of their cell module $S(\bla)$, thus determining the graded dimension of the simple modules. When $e=\infty$ or $e>n$, Hu and Mathas~\cite[Appendix~B]{HuMathas} give a formula for the graded decomposition numbers of the quiver Schur algebras of level two in terms of tableaux combinatorics; this is independent of the characteristic of the field. In current work~\cite{BDHS,BDDHMS}, the authors determine full submodule structures for the Specht modules of the core blocks of the Hecke algebras and the Weyl modules for their Schur algebras.

Each paper above uses a different notation. 
In this paper, we index the (bipartitions corresponding to) Specht modules in a core block by sign sequences, sequences containing elements from $\{0,-,+\}$ with a fixed number of entries of each type, where the $0$ entries are essentially redundant. The integers $n_B$ and $p_B$ above count the numbers of $-$s and $+$s respectively. The structure of the core blocks means that our combinatorics feels extremely natural. When $e=\infty$, it is straightforward to pass from our notation to that of~\cite{LM:FockI} or \cite{BDHS}. 

The structure of this paper is as follows. In Section~\ref{S:Background} we introduce the background material: Section~\ref{S:Abacus} defines the relevant combinatorics and Section~\ref{S:AK} introduces the Ariki-Koike algebras. For more information on the Ariki-Koike algebras and their combinatorics, we refer the reader to the survey paper of Mathas~\cite{Mathas:AKSurvey} and for more information on their connections with the cyclotomic KLR algebras of type $A$, we refer them to the survey paper of Kleshchev~\cite{Kleshchev:Survey}. 

In Section~\ref{S:r2}, we move on to considering the core blocks, as introduced by Fayers~\cite{Fayers:Cores}, when $r=2$. In Section~\ref{S:SS} we define the sign sequences and in Section~\ref{S:DecompMat} we show how they index the block decomposition matrices for core blocks. In Section~\ref{S:Fock} we briefly introduce the Fock space representation of $\mathcal{U}_v(\widehat{\mathfrak{sl}}_e)$ and describe its connection with the Ariki-Koike algebras; more details can be found in the papers of Lascoux, Leclerc and Thibon~\cite{LLT} and Ariki~\cite{Ariki}. In Section~\ref{S:Decomp}, we find the decomposition numbers for the core blocks. We start with the special case that the base tuple is flat, which includes the case that $e=\infty$, and then use what is essentially a focused version of Scopes equivalence~\cite{DA,Scopes} to generalize it to arbitrary core blocks. 
Finally, in Section~\ref{S:Wrapup} we summarize our results, present Witty's theorem and give some examples. 

\section{Background} \label{S:Background}
Throughout this paper, we will use a parameter $e \in \{2,3,\ldots\}\cup \{\infty\}$. If $e$ is finite, we  define $I=\{0,1,\ldots,e-1\}$ which we may identify with $\Z/e\Z$; otherwise we define $I=\Z$. In both cases we take $<$ to be the usual total order on $I$.

\subsection{Multipartitions and abacus configurations} \label{S:Abacus}
Suppose that $n \geq 0$. 
A partition of $n$ is a sequence $\la=(\la_1,\la_2,\ldots)$ of non-negative integers such that $\la_1 \geq \la_2 \geq \ldots$ and $\sum_{i \geq 1} \la_i =n$. We write $|\la|=n$. We let $\La_n$ denote the set of partitions of $n$ and $\La=\bigcup_{n \geq 0} \La_n$ denote the set of all partitions. Now suppose that $r \geq 1$. An $r$-multipartition, or multipartition, of $n$ is an $r$-tuple of partitions $\bla=(\la^{(1)},\la^{(2)},\ldots,\la^{(r)})$ such that $\sum_{k=1}^r |\la^{(k)}|=n$. We write $|\bla|=n$. We let $\La^{r}_n$ denote the set of $r$-multipartitions of $n$ and $\La^{r}=\bigcup_{n \geq 0} \La^{r}_n$ denote the set of all $r$-multipartitions.

Suppose that $\bla \in \La^{r}$. The Young diagram of $\bla$ is the set 
\[[\bla] = \{(x,y,k) \in \Z_{>0} \times \Z_{>0} \times \{1,2, \ldots, r\} \mid y \leq \la_x^{(k)}\}.\]
We say that a node $\mathfrak{n} \in [\bla]$ is removable if $[\bla] \setminus \{\mathfrak{n}\}$ is the Young diagram of a multipartition and we say that $\mathfrak{n} \notin [\bla]$ is addable if $[\bla] \cup \{\mathfrak{n}\}$ is the Young diagram of a multipartition. 
The rim of $[\bla]$ is the set of nodes $\{(x,y,k) \in [\bla] \mid (x+1,y+1,k) \notin [\bla]\}$. For $h \geq 1$, a removable $h$-rim hook is a set of $h$ connected nodes in the rim such that removing those nodes gives the Young diagram of a multipartition. 

Now fix $e \in \{2,3,\ldots\} \cup \{\infty\}$.  Suppose that $\mc \in \Z^r$. 
To each node $\mathfrak{n}=(x,y,k) \in  \Z_{>0} \times \Z_{>0} \times \{1,2 \ldots r\}$ we associate its residue $\res_{\mc}(\mathfrak{n}) \in I$. If $e$ is finite (resp. $e=\infty$) we set $\res_{\mc}(x,y,k) = a_k + y - x \mod e$ (resp. $\res_{\mc}(x,y,k)=a_k+y-x$). 
For $\bla \in \La^{r}$ we define the residue set of $\bla$ to be the multiset $\Res_{\mc}(\bla) = \{\res(\mathfrak{n}) \mid \mathfrak{n} \in [\bla]\}$. We define an equivalence relation $\sim_{\mc}$ on $\La^{r}$ by saying that $\bla \sim_{\mc} \bmu$ if and only if $\Res_{\mc}(\bla)=\Res_{\mc}(\bmu)$ and we refer to the $\sim_{\mc}$-equivalence classes of $\La^r$ as blocks. Clearly if $e$ is finite and $\ms \in \Z^r$ with $s_k \equiv a_k \mod e$ for $1 \leq k \leq r$ then $\bla \sim_{\mc} \bmu$ if and only if $\bla \sim_{\ms} \bmu$. 

Given $\mc \in \Z^r$, we define a subset $\La^{\mc}\subset \La^{r}$. For $i \in I$, we call an addable (resp. removable) node of residue $i$ an addable (resp. removable) $i$-node. 
Given $\bmu \in \La^{r}$ and $i \in I$, we define a total order $\lhd$ on the set of addable and removable nodes of $[\bmu]$ by saying that $(x_1,y_1,k_1) \lhd (x_2,y_2,k_2)$ if $k_1<k_2$ or if $k_1=k_2$ and $x_1 <x_2$. We may then define the $i$-signature of $\bmu$ by looking at all the addable and removable $i$-nodes of $[\bmu]$ ordered according to $\lhd$ and writing $\ab$ for an addable $i$-node and $\rr$ for a removable $i$-node. We then construct the reduced $i$-signature of $\bmu$ by repeatedly removing all adjacent $(\rr\ab)$-pairs until there are no such pairs left. If there are any $\rr$ terms in the reduced $i$-signature of $\bmu$, the first $\rr$ term corresponds to a removable $i$-node of $[\bmu]$ which is called a good $i$-node.

It is usual to write $+$ instead of $\ab$ and $-$ instead of $\rr$, however we want use those symbols for other purposes in this paper. The set $\La^{\mc}$ is defined recursively. 
  Suppose that $\bmu \in \La^{r}$.
  \begin{itemize}
  \item If $|\bmu|=0$ then $\bmu \in \La^{\mc}$. 
  \item Otherwise, if $\bmu$ does not contain a good $i$-node for any $i \in I$ then $\bmu \notin \La^{\mc}$.    
    \item Otherwise, suppose that $\mathfrak{n}$ is a good $i$-node of $\bmu$ for some $i \in I$. Let $\bar{\bmu}$ be the multipartition whose Young diagram is obtained from $[\bmu]$ by removing the node $\mathfrak{n}$. Then $\bmu \in \La^{\mc}$ if and only if $\bar{\bmu} \in \La^{\mc}$.   
    \end{itemize}

The multipartitions in $\La^{\mc}$ are known as Kleshchev multipartitions; they have the property that if $\bmu \in \La^{\mc}$ then $\mu^{(k)}$ is $e$-restricted for all $1 \leq k \leq r$. We set $\La^{\mc}_n = \La^{\mc} \cap \La^r_n$.

Given an $r$-multipartition, it is convenient to represent it as a $r$-tuple of abacus configurations. 
If $e$ is finite, an $e$-abacus is an abacus with $e$ vertical runners which are infinite in both directions and which are indexed from left to right by the elements of $I$. The possible bead positions are indexed by the elements of $\Z$ such that bead position $b$ on the abacus is in row $l$ of runner $i$ where $b=le+i$ and $i \in I$. If $e=\infty$, the abacus has runners and bead positions indexed by the elements of $I=\Z$, so that runner $i\in \Z$ contains either one bead or no beads.  

For $\la \in \La$ and $a \in \Z$, define the $\beta$-set \[B_a(\la)=\{\la_x - x + a \mid x \geq 1\}.\]
We define the abacus configuration of $\la$ with respect to $a$ to be the abacus where we put a bead at position $b$ for each $b \in B_a(\la)$. 
For $\bla \in \La^{r}$ and $\mc \in \Z^r$ we define the abacus configuration of $\bla$ with respect to $\mc$ to be an $r$-tuple of abacus configurations, the $k^{\text{th}}$ one of which is the abacus configuration of $\la^{(k)}$ with respect to $a_k$. 

Operations on the Young diagram of $\bla$ can be translated into operations on the abacus configuration of $\bla$. 

\begin{lemma}
Let $\bla \in \La^r$ and $\mc \in \Z^r$. All residues below are with repect to $\mc$ as are all abacus configurations.  
\begin{itemize}
\item Removing an $i$-node from component $k$ of $[\bla]$ corresponds to moving a bead on runner $i$ of the abacus of $\la^{(k)}$ back by one position.
\item Adding an $i$-node to component $k$ of $[\bla]$ corresponds to moving a bead on runner $i-1$ of the abacus of $\la^{(k)}$ forward by one position.
\item If $e \ne \infty$ then adding (resp. removing) an $e$-rim hook to (resp. from) component $k$ of $[\bla]$ corresponds to pushing a bead on the abacus of $\la^{(k)}$ down one position (resp. up one position). 
\end{itemize}
\end{lemma}

\begin{ex}
Take $e=3$.  Let $\bla=((6,5,2,1^2),(6,1))$ and take $\mc=(8,6)$. Then the abacus configuration of $\bla$ with respect to $\mc$ is given by
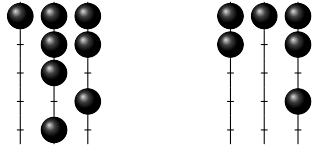
\begin{center}
  \abacus(bbb,nbb,nbn,nnb,nbn) \qquad \qquad \abacus(bbb,bnb,nnn,nnb,nnn)
\end{center}
and $[\bla]$ has three removable $e$-rim hooks. For example, $[((4^2,2,1^1),(6,1))]$ is formed by $[\bla]$ by removing an $e$-rim hook from $[\la^{(1)}]$ and has abacus configuration (with respect to $\mc$)
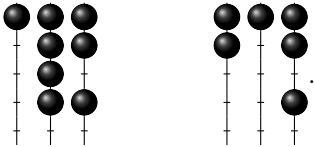
\begin{center}
  \abacus(bbb,nbb,nbn,nbb,nnn) \qquad \qquad \abacus(bbb,bnb,nnn,nnb,nnn).
\end{center}
When drawing abacuses we truncate the runners and assume that the positions above the drawn portion all contain beads while the positions below are empty. 
\end{ex}

We say that $\bla \in \La^r$ is a multicore if $e=\infty$ or if $e < \infty$ and no component of $\bla$ has any removable $e$-rim hooks; equivalently, $\bla$ is a multicore if no bead in the abacus of any component has an empty space above it. (This is independent of the choice of $\mc$ used to draw the abacus configurations.) We say that a $\sim_{\mc}$-equivalence class $B$ is a core block if every $\bla \in B$ is a multicore. 
If $\bla$ is a multicore and $\mc \in \Z^r$ then for $1 \leq k \leq r$ and $i \in I$ we define $l^{\mc}_{ik}(\bla)$ as follows. If $e$ is finite, set
\[l^{\mc}_{ik}(\bla) = \max\{l \in \Z \mid le+i \in B_{a_k}(\la^{(k)})\};\]
in other words, $l^{\mc}_{ik}(\bla)$ is the lowest row of the abacus configuration for $\la^{(k)}$ with respect to $a_k$ which contains a bead on runner $i$. If $e=\infty$, set $l^{\mc}_{ik}(\bla)$ to be $1$ if $i \in B_{a_k}(\la^{(k)})$ and $0$ otherwise. 

\begin{lemma}[{\cite[Theorem~3.1]{Fayers:Cores}}] \label{L:BaseTuple}
Let $\ms \in \Z^r$ and suppose that $B$ is a $\sim_{\ms}$-equivalence class.
\begin{enumerate}
\item Suppose $e$ is finite. If $B$ is a core block then there exists $\mc=(a_1,a_2,\ldots,a_r) \in \Z^r$, 
with $a_k \equiv s_k \mod e$ for all $1 \leq k \leq r$, 
and $\bb =(b_0,b_1,\ldots,b_{e-1}) \in \Z^e$ such that for each $i \in I$, $1 \leq k \le r$ and $\bla \in B$, $l^{\mc}_{ik}(\bla)$ is equal to either $b_i$ or $b_i+1$. 
\item Suppose $e=\infty$. Then every block is a core block and if we set $\bb=(\ldots,0,0,0,\ldots)$ and take $\mc=\ms$ then for each $i \in I$, $1 \leq k \le r$ and $\bla \in B$, $l^{\mc}_{ik}(\bla)$ is equal to either $b_i$ or $b_i+1$.  
\end{enumerate}
Conversely, suppose that $e$ is finite and that $\bla \in \La^r$. If there exists $\mc=(a_1,a_2,\ldots,a_r) \in \Z^r$, 
with $a_k \equiv s_k \mod e$ for all $1 \leq k \leq r$, 
and $\bb =(b_0,b_1,\ldots,b_{e-1}) \in \Z^e$ such that for each $i \in I$ and $1 \leq k \le r$, $l^{\mc}_{ik}(\bla)$ is equal to either $b_i$ or $b_i+1$, then the $\sim_{\ms}$-equivalence class of $\bla$ is a core block.      
\end{lemma}

\subsection{The Ariki-Koike algebra} \label{S:AK}
Let $r\geq 1$ and $n \geq 0$ and let $\mathbb{F}$ be a field of characteristic $p \geq 0$. Choose $q \in \mathbb{F} \, \setminus \, \{0\}$ and ${\bf Q}=(Q_1,\dots,Q_r) \in \mathbb{F}^r$. The Ariki-Koike algebra $\h=\h_{r,n}(q,\bf Q)$ is the unital associative $\mathbb{F}$-algebra with generators $T_0 , \dots, T_{n-1}$ and relations
$$\begin{array}{crcll}
& (T_i +q )(T_i -1) & = & 0, &  \text{ for } 1 \leq i \leq n-1, \\
& T_i T_j & = & T_j T_i, & \text{ for } 0 \leq i,j \leq n-1, |i-j|>1, \\
& T_i T_{i+1} T_i & = & T_{i+1} T_i T_{i+1}, & \text{ for } 1 \leq i \leq n-2,\\
& (T_0 - Q_1)\dots (T_0 - Q_r) & = & 0, & \\
& T_0 T_1 T_0 T_1 & = & T_1 T_0 T_1 T_0. &
\end{array}$$
Define $e \geq 2$ to be minimal such that $1+q+\dots+q^{e-1}=0$, or set $e=\infty$ if no such value exists. 
Two parameters $Q_k$ and $Q_l$ are $q$-connected if $Q_k = q^a Q_l$ for some $a \in I$. Each Ariki-Koike algebra~$\h$ is Morita equivalent to a direct sum of tensor products of smaller algebras whose parameters are all $q$-connected~\cite{DM:Morita} and so we will assume that all our parameters are $q$-connected, in fact, that they are all powers of $q$ where $q\neq 1$. If $\mc = (a_1,a_2,\ldots,a_r) \in \Z^r$ satisfies $Q_k = q^{a_k}$ for all $1 \leq k \leq r$ then we call $\mc$ a multicharge for $\h$. If $e$ is finite then $q^e=1$ so there are infinitely many possible multicharges for $\h$. 

The algebra $\h$ is a cellular algebra~\cite{GL,DJM:CellularBasis} with the cell modules indexed by the $r$-multipartitions of $n$. The cell module $S^{\bla}$ indexed by the multipartition $\bla$ is called a Specht module. Due to the properties of cellular algebras, all the composition factors of $S^{\bla}$ lie in the same block, and so we can think of the Specht modules as being partitioned into blocks. 

\begin{proposition}[{\cite[Theorem~2.11]{LM:Blocks}}] \label{P:Delta}
Suppose that $\mc$ is a multicharge for $\h$ and that $\bla, \bmu \in \La^{r}_n$. Then $S^{\bla}$ and $S^{\bmu}$ lie in the same block of $\h$ if and only if $\Res_{\mc}(\bla)=\Res_{\mc}(\bmu)$. 
\end{proposition}

This result explains why we called the $\sim_{\mc}$-equivalence classes blocks: two multipartitions of $n$ are in the same class if and only if the corresponding Specht modules lie in the same block. 

From the properties of cellular algebras, we know that there is a bilinear form on each cell module. If we define $\rad(S^{\bmu})$ to be the radical of the Specht module $S^{\bmu}$ with respect to this bilinear form then if $\mc$ is a multicharge for $\h$ then $\rad(S^{\bmu}) \neq S^{\bmu}$ if and only if $\bmu \in \La^{\mc}_n$ and so
\[\{D^{\bmu} = S^{\bmu}/\rad(S^{\bmu}) \mid \bmu \in \La^{\mc}_n\}\]
is a complete set of non-isomorphic irreducible $\h$-modules. Given $\bla \in \La^r_n$ and $\bmu \in \La^{\mc}_n$ we define $[S^{\bla}:D^{\bmu}]$ to be the multiplicity of the simple module $D^{\bmu}$ as a composition factor of the Specht module $S^{\bla}$. 

Brundan and Kleshchev~\cite{BK:Blocks} have shown that the Ariki-Koike algebras are isomorphic to certain graded algebras defined by Khovanov and Lauda~\cite{KhovLaud:diagI,KhovLaud:diagII} and by Rouquier~\cite{Rouquier}: the cyclotomic KLR algebras of type $A$. Through this isomorphism, we may think of $\h$ as being graded. There is a grading on the Specht modules~\cite{BKW}, thus we can define $[S^{\bla}:D^{\bmu}]_v \in \N[v,v^{-1}]$ to be the graded multiplicity of the simple module $D^{\bmu}$ as a composition factor of the Specht module $S^{\bla}$; we recover the original decomposition number by setting $v=1$. For more details, we refer the reader to the survey paper~\cite{Kleshchev:Survey}. 

We define the (graded) decomposition matrix of $\h$ to be the matrix whose rows are indexed by the elements of $\La^r_n$ and whose columns are indexed by the elements of $\La^{\mc}_n$ with entries equal to $[S^{\bla}:D^{\bmu}]_v$, for $\bla \in \La^r_n$ and $\bmu \in \La^{\mc}_n$. If $B$ is a $\sim_{\mc}$-equivalence class, where $\bla \in B$ is such that $|\bla|=n$, we define the block decomposition matrix of $B$ to be the submatrix of the decomposition matrix of $\h$ whose rows and columns are indexed only by elements of $B$. 

\section{Core blocks when $r=2$} \label{S:r2}
For the rest of this paper, we assume $r=2$ so that our Specht modules are indexed by bipartitions and we shall refer to multicores as bicores. 
We fix $e \in \{2,3,\ldots\} \cup \{\infty\}$ and $\ms =(s_1,s_2) \in I^2$. We shall take a block to mean an $\ms$-equivalence class of $\La^2$ and a core block to be a block $B$ in which each $\bla \in B$ is a bicore. 

Let $\mathbb{F}$ be a field of characteristic $p \geq 0$. Fix $q \in \mathbb{F}$ where $q$ is a primitive $e^{\text{th}}$ root of unity if $e$ is finite and $q^l \neq 0,1$ for any $l \in \Z$ otherwise. Let ${\bf Q}=(q^{s_1},q^{s_2})$. 
If $B$ is a block then there exists $n\geq 0$ such that $|\bla|=n$ for all $\bla \in B$ and so we associate to $B$ the block of the Hecke algebra $\h_n=\h_{2,n}(q,{\bf Q})$ containing the Specht modules $S^{\bla}$ for $\bla \in B$. We denote this corresponding block of $\h_n$ by $\hat{B}$ so that we have a correspondence $B \leftrightarrow \hat{B}$ between $\ms$-equivalence classes and $2$-sided indecomposable ideals of the Hecke algebras $\h_n$; we refer to both as blocks. 

\subsection{Sign sequences} \label{S:SS}
 Set 
\[\Delta = \{\delta=  (\delta_i)_{i \in I} \mid \delta_i \in \{-,0,+\} \text{ and if } e=\infty \text{ then } \delta_i = \delta_{-i}=0 \text{ for all } i \gg 0\}.\]
We say that $\delta,\delta' \in \Delta$ are essentially the same if the two sequences obtained by removing all the $0$s from each of them are the same. 

For $\delta \in \Delta$, we define sets $\R(\delta) \in \Delta$ and $\St(\delta) \subseteq I^2$ as follows. 

\begin{enumerate}
\item Take $\R(\delta)=\delta$ and $\St(\delta)=\emptyset$. 
\item If there do not exist $i,j \in I$ with $i < j$ and $\delta_i=-$ and $\delta_j=+$, end the process. Return $\R(\delta)$ and $\St(\delta)$. 
\item Otherwise, choose $i,j \in I$ with $i < j$ and $\delta_i=-$ and $\delta_j=+$ with the property that $\delta_m=0$ for all $i < m < j$. 
Add $(i,j)$ to $\St(\delta)$ and set $\R(\delta)_i = \R(\delta)_j = 0$. Go back to step (2). 
\end{enumerate}

\begin{ex}
Suppose that $e=19$ and $\delta = (-,0,-,0,+,-,0,0,+,-,+,+,0,+,-,0,+,0,+)$. We find it helpful to draw a diagram for $\delta$ as below. 

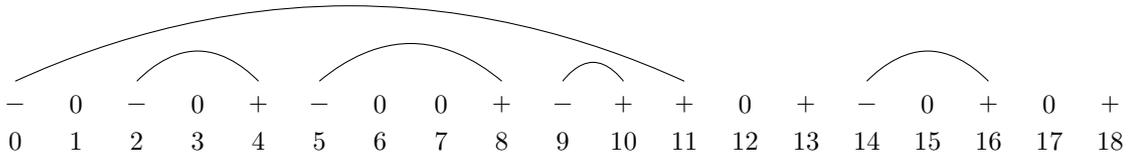
\begin{center}
\begin{tikzpicture}[xscale=.8]
\foreach \k in {0,2,5,9,14} {\node at (\k,1.7){$-$};}
\foreach \k in {1,3,6,7,12,15,17} {\node at (\k,1.7){$0$};}
\foreach \k in {4,8,10,11,13,16,18} {\node at (\k,1.7){$+$};}
\foreach \k in {0,1,...,18} {\node at (\k,1.2){$\k$};}
\draw (2,2) parabola[parabola height=.4cm] +(2,0);
\draw (5,2) parabola[parabola height=.5cm] +(3,0);
\draw (9,2) parabola[parabola height=.25cm] +(1,0);
\draw (14,2) parabola[parabola height=.4cm] +(2,0);
\draw (0,2) parabola[parabola height=1cm] +(11,0);
\end{tikzpicture}
\end{center}

Then 
\begin{align*}
\St(\delta) & = \{(2,4), (5,8), (9,10),(0,11),(14,16)\}, \\
\R(\delta) & = (0,0,0,0,0,0,0,0,0,0,0,0,0,+,0,0,0,0,+).
\end{align*}
\end{ex}
Set
\[\Delta_0 = \{ \delta \in \Delta \mid - \notin \R(\delta) \text{ or } + \notin \R(\delta)\}.\]
Suppose $\delta \in \Delta_0$. For each $S = \{(i_1,j_1),\ldots,(i_t,j_t)\} \subseteq \St(\delta)$, define $\delta^S \in \Delta$ by setting  
\[\delta^S_m = \begin{cases} +, & m=i_l \text{ for some } 1 \leq l \leq t, \\
-, & m=j_l \text{ for some } 1 \leq l \leq t, \\
\delta_m, & \text{otherwise}.
\end{cases}\]

If $\delta'=\delta^S$ for some $S \subseteq \St(\delta)$, we write $\delta' \ra \, \delta$ and set
$\ell(\delta',\delta) = |S|$. We write $\delta' \rra \, \delta$ if  $\delta'=\delta^S$ for some $S \subseteq \St(\delta)$ where $S$ has the property that if $(i_1,j_1),(i_2,j_2) \in \St(\delta)$ with $i_1 < i_2 < j_2 < j_1$ then $(i_1,j_1) \in S \implies (i_2,j_2) \in S$.

\begin{ex}
Let $e=8$ and $\delta = (-,-,+,-,-,+,+,-) \in \Delta_0$. We write $\delta$ as 
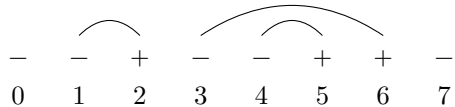
\begin{center}
\begin{tikzpicture}[xscale=0.8]
\foreach \k in {0,1,3,4,7} {\node at (\k,1.7){$-$};}
\foreach \k in {2,5,6} {\node at (\k,1.7){$+$};}
\foreach \k in {0,1,...,7} {\node at (\k,1.2){$\k$};}
\draw (1,2) parabola[parabola height=.2cm] +(1,0);
\draw (4,2) parabola[parabola height=.2cm] +(1,0);
\draw (3,2) parabola[parabola height=.4cm] +(3,0);
\end{tikzpicture}
\end{center}
so that $\St(\delta)=\{(1,2),(4,5),(3,6)\}$. Then we have 
\[\begin{array}{|c|c|c|c|} \hline 
\raisebox{-2pt}{$S$} & \raisebox{-2pt}{$\delta^S$} & \raisebox{-2pt}{$\ell(\delta^S,\delta)$} & \raisebox{-2pt}{$\delta^S \rra \, \delta$} \\ [2pt] \hline 
\emptyset & (-,-,+,-,-,+,+,-) & 0 & \checkmark \\
\{(1,2)\} & (-,+,-,-,-,+,+,-) & 1 & \checkmark  \\
\{(4,5)\} & (-,-,+,-,+,-,+,-) & 1  & \checkmark  \\
\{(3,6)\} & (-,-,+,+,-,+,-,-) & 1 & \quad \\ 
\{(1,2),(4,5)\} & (-,+,-,-,+,-,+,-) & 2 &  \checkmark \\
\{(1,2),(3,6)\} & (-,+,-,+,-,+,-,-) & 2 &\quad \\
\{(4,5),(3,6)\} & (-,-,+,+,+,-,-,-) & 2  &  \checkmark \\
\{(1,2),(4,5), (3,6)\} & (-,+,-,+,+,-,-,-) & 3 &  \checkmark \\ \hline
\end{array}\]
\end{ex}

Informally, we can see that each element of the set $\{\delta^{S} \mid S \subseteq \St(\delta)\}$ is obtained by swapping some pairs $-+$ that lie at either end of an arc in the diagram of $\delta$. 

\subsection{Indexing the block decomposition matrices corresponding to core blocks} \label{S:DecompMat}
Let $B$ be a core block. Suppose $e$ is finite. By Lemma~\ref{L:BaseTuple}, we can find $\mc=(a_1,a_2) \in \Z^2$ and $\bb=(b_0,b_1,\ldots,b_{e-1}) \in \Z^e$ with the property that $a_k \equiv s_k \mod e$ for $k=1,2$ and for any $\bla \in B$ we have $l^{\mc}_{ik}(\bla)=b_i$ or $l^{\mc}_{ik}(\bla)=b_i+1$, for $i \in I$ and $k=1,2$. If $\mc,\bb$ satisfy these conditions we call $\re=(\mc,\bb)$ a reduced pair for $B$.  We define a total order $\prec$ (which depends on the choice of $\bb$) on $I$ by saying that $i \prec j$ if $b_i < b_j$ or if $b_i = b_j$ and $i<j$. 
If $I=\{i_0,i_1,\ldots,i_{e-1}\}$ with $i_0 \prec i_1 \prec \ldots \prec i_{e-1}$ we take $\pi=\pi_{\bb}(B)$ be the permutation that sends $j$ to $i_j$ for $j \in I$. 
 If $e=\infty$, we define $\re=(\ms,\bz)$ where $\bz=(\ldots,0,0,0,\ldots)$ to be the unique reduced pair for $B$. We then take $\prec$ to be the usual total order $<$ on $\Z$ and $\pi$ to be the identity permutation on $\Z$. 

For $\bla \in B$ and $\re=(\mc,\bb)$ a reduced pair for $B$, define $\de^{\re}_{\bla} = ((\delta^{\re}_{\bla})_i)_{i \in I}$ by setting
\[(\delta^{\re}_{\bla})_i = l^{\mc}_{\pi(i)2}(\bla)-l^{\mc}_{\pi(i)1}(\bla)\]
for all $i \in I$. Our choice of $\re$ ensures that $\delta^{\re}_{\bla} \in \Delta$, where we abuse notation by identifying $1$ with $+$ and $-1$ with $-$. 

If $e$ is finite, the reduced pair $\re$ is not uniquely determined by the conditions above. However we shall see that unless $|B|=1$, $\delta^{\re}_{\bla}$ and $\delta^{\re'}_{\bla}$ are essentially the same for any reduced pairs $\re$ and $\re'$. 

\begin{ex} Take $e=7$ and $\ms=(1,6)$. Let $B$ be the core block containing the bipartition $$\bla=((13,10,8,7,6,4,3,2,1^6),((13,12,10,9,8,6,5^3,3,2^3,1^6)).$$ 
If $\re=(\mc,\bb)$ is a reduced pair for $B$ then there exists $l \in \Z$ such that $\mc$ is of the form $\mc=(8+7l,6+7l)$. 
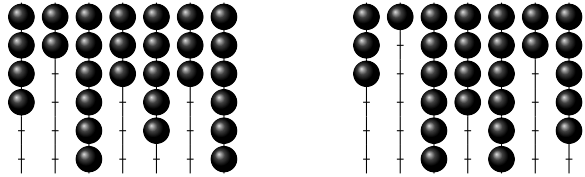
\begin{center}
\abacus(bbbbbbb,bbbbbbb,bnbbbbb,bnbnbnb,nnbnbnb,nnbnnnb)
\qquad \qquad 
\abacus(bbbbbbb,bnbbbbb,bnbbbnb,nnbbbnb,nnbnbnb,nnbnbnn)
\end{center}

Taking $\mc=(29,27)$ we have two choices for $\bb$, each of which gives a different ordering $\prec$: 
\begin{align*}
\bb & =(3,1,6,3,5,2,5) && \implies 1 \prec 5 \prec 0 \prec 3 \prec 4 \prec 6 \prec 2 & \implies \delta^{\re}_{\bla} &= (-,-,-,+,+,-,0), \\
\bb& =(3,1,5,3,5,2,5) && \implies 1 \prec 5 \prec 0 \prec 3 \prec 2 \prec 4 \prec 6 & \implies \delta^{\re}_{\bla} &= (-,-,-,+,0,+,-). 
\end{align*}
However, the only choice comes when we look at the third runner on the abacus configurations, that is, when $l^{\mc}_{22}(\bla)=l^{\mc}_{21}(\bla)=6$, and we note that the two choices for $\delta^{\re}_{\bla}$ are essentially the same. 
\end{ex}

\begin{lemma}[{\cite[Propn.~3.7]{Fayers:Cores}}] \label{L:GenBlock}
Suppose $B$ is a core block and $\re=(\mc,\bb)$ is a reduced pair for $B$. Let $\bla \in B$ and suppose
$i,j \in I$ with $(\delta^{\re}_{\bla})_i=-$ and $(\delta^{\re}_{\bla})_j=+$.
Define $s_{ij}(\bla)$ to be the bicore $\bmu$ with
\begin{align*}
l^{\mc}_{mk}(\bmu) & = \begin{cases}
l^{\mc}_{mk}(\bla) - 1, & \pi(m)=i \text{ and } k=1 \text{ or } \pi(m)=j \text{ and } k=2, \\
l^{\mc}_{mk}(\bla)+1, & \pi(m)=j \text{ and } k=1 \text{ or } \pi(m)=i \text{ and } k=2,\\
l^{\mc}_{mk}(\bla), & \text{otherwise}.
\end{cases}
\intertext{Then $\bmu \in B$ and} 
 (\delta^{\re}_{\bmu})_m & = \begin{cases}
+, &  m=i,\\
-, & m=j,\\
(\delta^{\re}_{\bla})_m, & \text{otherwise}.
\end{cases}
\end{align*}
Moreover, if $\bnu \in B$ then we may form $\bnu$ from $\bla$ by repeatedly applying operations of the form $s_{ij}$ for some $i,j$ as above. 
\end{lemma}

\begin{corollary} \label{C:GenBlock}
Suppose that $B$ is a core block and $\re$ is a reduced pair for $B$. Let $\bla,\bmu \in B$. Then $\delta^{\re}_{\bmu}$ can be formed from $\delta^{\re}_{\bla}$ by permuting the entries equal to $\pm$. Conversely, any sequence $\epsilon \in \Delta$ formed by permuting the entries equal to $\pm$ in $\delta^{\re}_{\bla}$ is equal to $\delta^{\re}_{\bnu}$ for some $\bnu \in B$. 

Consequently, $|B|=1$ if and only if $- \notin \delta^{\re}_{\bla}$ or $+ \notin \delta^{\re}_{\bla}$.  
\end{corollary}

\begin{lemma} \label{L:DeltaSame}
Suppose $e$ is finite and that $B$ is a core block with $|B|>1$.  Suppose that $\re=(\mc,\bb)$ and $\re'=(\mc',\bb')$ are reduced pairs for $B$ with $\mc=(a_1,a_2)$ and $\mc'=(a_1',a_2')$. Let $\bla \in B$.
\begin{enumerate}
\item There exists $d\in \Z$ with $a'_k = a_k +de$ for $k=1,2$.  
\item The sequences $\delta^{\re}_{\bla}$ and $\delta^{\re'}_{\bla}$ are essentially the same.
\end{enumerate}
\end{lemma}

\begin{proof}
\quad 
\begin{enumerate}
\item Since $\re$ and $\re'$ are both reduced pairs for $B$, there exist $d_1,d_2 \in \Z$ such that $a_k'=a_k + d_ke$ for $k=1,2$. As $|B|>1$, there exist $i,j \in I$ such that 
\begin{align*}
& l^{\mc}_{i2}(\bla) - l^{\mc}_{i1}(\bla) = -1 && \implies & & l^{\mc'}_{i2}(\bla) - l^{\mc'}_{i1}(\bla) = d_2-d_1 -1, \\ 
& l^{\mc}_{j2}(\bla) - l^{\mc}_{j1}(\bla) = 1 && \implies && l^{\mc'}_{j2}(\bla) - l^{\mc'}_{j1}(\bla) = d_2-d_1+1.
\end{align*} 
so since $|l^{\mc'}_{m2}(\bla) - l^{\mc'}_{m1}(\bla)|\leq 1$ for all $m \in I$, we must have $d_1=d_2$. 
\item 
Using part (1) we may assume that $a'_k=a_k+de$ for $k=1,2$. If we set ${\mathbf c}=(b_0+d,b_1+d,\ldots,b_{e-1}+d)$ then $\hat{\re}=(\mc',{\mathbf c})$ is a reduced pair for $B$ and $\delta^{\hat{\re}}_{\bla}=\delta^{\re}_{\bla}$. 
Let
\[I^{\pm} = \{i \in I \mid  l^{\mc'}_{i2}(\bla) \neq l^{\mc'}_{i1}(\bla)\},\]
so that $(\delta^{\hat{\re}}_{\bla})_i \ne 0$ if and only if $i \in I^{\pm}$. 
If $i \in I^{\pm}$ then we must have $b'_i = \min\{l^{\mc'}_{i2}(\bla),l^{\mc'}_{i1}(\bla) \}=c_i$. So restricting the order $\prec$ to the set $I^{\pm}$ we see that changing $(\mc',{\mathbf c})$ to $(\mc',\bb')$ permutes the elements of 
$\delta^{\hat{\re}}_{\bla}$ while keeping the non-zero entries in the same order. 
Hence $\delta^{\hat{\re}}_{\bla}$ and $\delta^{\re'}_{\bla}$ are essentially the same. 
\end{enumerate}
\end{proof}

If $|B|=1$, Lemma~\ref{L:DeltaSame} does not hold. If $e$ is finite and $B=\{\bla\}$, assume that we choose the reduced pair $\re$ such that $0 \in \delta^{\re}_{\bla}$ and $1 \notin \delta^{\re}_{\bla}$. This choice is arbitrary; it is simply made so that for every element of a core block we have an expression $\delta^{\re}_{\bla}$ which is essentially independent of the choice of the reduced pair $\re$. Recall that if $e=\infty$, there is a unique reduced pair for $B$.
Unless we need to emphasise the reduced pair, we will henceforth write $\delta_{\bla}$ instead of $\delta^{\re}_{\bla}$. 

Suppose that $B$ is a  core block. Take $\bla \in B$ and hence define 
\[n_B=\#\{i \in I \mid (\delta_{\bla})_i = -\}, \qquad p_B=\#\{i \in I \mid (\delta_{\bla})_i = +\}, \qquad m_B = \min\{n_B,p_B\}.\]
Given Corollary~\ref{C:GenBlock} and the discussion above, these parameters are well-defined and independent of the choice of $\bla$. 

\begin{lemma}
There exists $d \in \Z$ such that if $e$ is finite (resp. $e=\infty$) then $s_1 \equiv n_B +d \mod e$ and $s_2 \equiv p_B+d \mod e$ (resp. $s_1=n_B+d$ and $s_2=p_B+d$).\end{lemma}

\begin{proof}
Suppose that $\re=(\mc,\bb)$ is a reduced pair for $B$ and let $\bla \in B$. 
Suppose $e$ is finite. Then by~\cite[Lemma~2.2]{L:Rouquier}, $s_k \equiv a_k \equiv \sum_{i \in I} l^{\mc}_{ik}(\bla) \mod e$ for $k=1,2$. Hence
\[s_2 - s_1 \equiv \sum_{i \in I} (l^{\mc}_{i2}(\bla) - l^{\mc}_{i1}(\bla)) \equiv \sum_{i \in I} (\delta_{\bla})_i \equiv p_B-n_B \mod e.\]
Suppose $e=\infty$. Choose $t\ll 0$. Then $s_k=a_k=t+\#\{i \in I \mid i \geq t \text{ and } l^{\mc}_{ik}(\bla)=1\}$ for $k=1,2$. Hence
\[s_2-s_1 = \sum_{i \geq t} (l^{\mc}_{i2}(\bla)-l^{\mc}_{i1}(\bla)) = \sum_{i \geq t} (\delta_{\bla})_i = p_B-n_B.\]
\end{proof}

In~\cite[Section~2.1]{Fayers:Weights}, Fayers introduced the weight of a multipartition. Two multipartitions in the same block have the same weight and so for a block $B$ we define $\wt(B)$ to be the weight of any multipartition belonging to $B$. We define the the weight of the corresponding block $\hat{B}$ as $\wt(\hat{B})=\wt(B)$. We do not give Fayers' definition here, but we note that, roughly speaking, the blocks $\hat{B}$ of small weight tend to be easier to understand. We have $|B|=1$ if and only if $\wt(B)=0$ which holds if and only if $\hat{B}$ is simple.  

\begin{lemma}[{\cite[Propn.~3.8]{Fayers:Weights}}] \label{L:Weight2}
Suppose that $B$ is a core block. 
Then $\wt(B)=m_B$.
\end{lemma}

\begin{corollary}
Let $B$ be a core block. If $e$ is finite then \[0 \leq \wt(B) \leq \left\lfloor \frac{e}{2} \right \rfloor.\] 
\end{corollary}

For a core block $B$, we have seen that the Specht modules in $\hat{B}$ can be indexed by a subset of $\Delta$. The next problem is to decide which $\bmu \in B$ index simple modules. We will see that this can be determined by looking at $\delta_{\bmu}$. Let $\bz = (\bzi)_{i \in I}$ be the sequence where $\bzi=0$ for all $i \in I$.
We first show that it is sufficient to consider the case where $\re=(\mc,\bz)$ is a reduced pair for $B$.  Note that in this case, we have $\prec$ equal to the usual total order $<$ on $I$.

\begin{lemma}[{\cite[Proposition 2.10]{LyleRuff}}] \label{L:KleshReduce}
Suppose that $B$ is a core block with $(\mc,\bb)$ a reduced pair for $B$. 
Let $\bmu \in B$. Then for $i \in I$ and $k=1,2$ we have $l^{\mc}_{ik}(\bmu)=b_i + x_{ik}$ where $x_{ik} \in \{0,1\}$. 
 We define $\bar{\bmu}\in \La^{2}$ and $\bar{\mc}\in \Z^2$ so that 
\[l^{\bar{\mc}}_{ik}(\bar{\bmu})=x_{ik}\]
for $i \in I$ and $k=1,2$; note that this abacus configuration does uniquely define both $\bar{\mc}$ and $\bar{\bmu}$ and that $\delta_{\bmu}^{\mc}=\delta_{\bar{\bmu}}^{\bar{\mc}}$, so that if $\bar{B}$ is the $\bar{\mc}$-equivalence class containing $\bar{\bmu}$ then $(\bar{\mc},\bz)$ is a reduced pair for $\bar{B}$.  
Then $\bmu \in \La^{\ms}$ if and only if $\bar{\bmu} \in \La^{\bar{\mc}}$. 
\end{lemma}

\begin{lemma} \label{L:StillKlesh}
 Suppose that $B$ is a core block and $\re=(\mc,\bz)$ is a reduced pair for $B$. Let $\bmu \in B$. 
Suppose that $i-1,i \in I$ with $i-1<i$ and let $\bnu$ be the bipartition whose Young diagram is formed by removing all removable $i$-nodes from $[\bmu]$. Then unless $(\delta_{\bmu})_{i-1}=+$ and $(\delta_{\bmu})_{i}=-$ we have that $\bmu \in \La^{\ms}$ if and only if $\bnu \in \La^{\ms}$. 
\end{lemma}

\begin{proof}
Since $\re=(\mc,\bz)$, each component of $[\bmu]$ has at most one removable $i$-node. If $[\bmu]$ has no removable $i$-nodes, the lemma is trivial. So suppose that $[\bmu]$ has one removable $i$-node. Then the $i$-residue sequence of $\bmu$ is either $\rr$, $\ab\rr$ or $\rr\ab$. In the first two cases, this removable node is good, and so the lemma holds. In the third case, the first component of $[\bmu]$ has a removable $i$-node and the second has an addable $i$-node. Given $\re=(\mc,\bz)$, this is only  possible if $(\delta_{\bmu})_{i-1}=+$ and $(\delta_{\bmu})_i = -$. 

Now suppose that $[\bmu]$ has two removable $i$-nodes so the $i$-residue sequence of $\bmu$ is $\rr\rr$ and the removable $i$-node in the first component is good. Removing this node gives a bipartition, ${\bsig}$ say, with $\bmu \in \La^{\ms}$ if and only if ${\bsig} \in \La^{\ms}$. By the last paragraph, the one removable $i$-node in $[\bsig]$ is good, so that ${\bsig} \in \La^{\ms}$ if and only if $\bnu \in \La^{\ms}$. 
\end{proof}

\begin{lemma} \label{L:LoseZeros}
Suppose that $B$ is a core block with $|B|>1$ and $(\mc,\bz)$ is a reduced pair for $B$. Let $\bmu \in B$. There is a bipartition $\bnu$ which satisfies the following conditions. 
\begin{itemize}
\item $\bnu$ lies in a core block $B'$ and $(\mc,\bz)$ is a reduced pair for $B'$.
\item $\delta_{\bnu}$ is essentially the same as $\delta_{\bmu}$. 
\item There exist $i_1,i_2 \in I$ with $i_1\leq i_2$ such that
\begin{itemize}
\item For $k=1,2$, $l^{\mc}_{mk}(\bnu)=1$ for all $m < i_1$ and $l^{\mc}_{mk}(\bnu)=0$ for all $m > i_2$.
\item $(\delta_{\bnu})_m \in \{-,+\}$ for $i_1\leq m \leq i_2$. 
\end{itemize}
\end{itemize}
Then $\bmu \in \La^{\ms}$ if and only if $\bnu \in \La^{\ms}$. 
\end{lemma}

Before proving this result, we give an example of bipartitions $\bmu$ and $\bnu$ which satisfy the lemma. 

\begin{ex} Below, we have $\delta_{\bmu}=(+,0,+,0,-,-,+,0,+,0,-)$ and $\delta_{\bnu}=(0,0,+,+,-,-,+,+,-,0,0)$. 
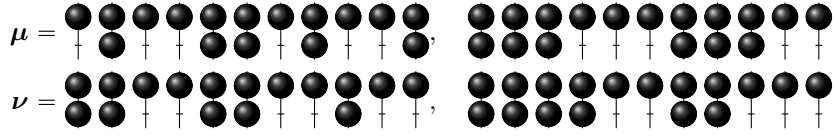
\begin{align*}
\bmu & = \abacus(bbbbbbbbbbb,nbnnbbnbnnb), \quad  \abacus(bbbbbbbbbbb,bbbnnnbbbnn) \\ 
\bnu & = \abacus(bbbbbbbbbbb,bbnnbbnnbnn), \quad  \abacus(bbbbbbbbbbb,bbbbnnbbnnn)
\end{align*}
\end{ex}

\begin{proof}[Proof of Lemma~\ref{L:LoseZeros}]
Suppose there exist $i,j \in I$ with $j<i$ and $l^{\mc}_{ik}(\bmu)=1$ for $k=1,2$ while $l^{\mc}_{jk}(\bmu)=0$ for $k=1$ or $k=2$.  
Choose $i$ minimal with this condition. Then following Lemma~\ref{L:StillKlesh}, we may remove all the removable $i$-nodes of $[\bmu]$ to obtain the Young diagram of a bipartition ${\tilde{\bsig}}$ with $\tilde{\bsig} \in \La^{\ms}$ if and only if $\bmu \in \La^{\ms}$. Then the abacus configuration of $\tilde{\bsig}$ is formed from that of $\bmu$ by swapping runners $i-1$ and $i$, so that $\delta_{\tilde{\bsig}}$ is formed from $\delta_{\bmu}$ by swapping the entries in the positions indexed by $i-1$ and $i$; note that $(\delta_{\bmu})_i=0$ so that $\delta_{\bmu}$ and $\delta_{\tilde{\bsig}}$ are essentially the same. We continue in this way until we reach a bipartition $\bsig$ where there do not exist $i,j$ as above. 

Now consider $\bsig$. Suppose there exist $i',j' \in I$ with $i'<j'$ and $l^{\mc}_{i'k}(\bsig)=0$ for $k=1,2$ while $l^{\mc}_{j'k}(\bsig)=1$ for $k=1$ or $k=2$.  
Choose $i'$ maximal with this property. Again following Lemma~\ref{L:StillKlesh}, we may remove all the removable $(i'+1)$-nodes of $[\bsig]$ to obtain the Young diagram of a bipartition $\tilde{\bnu}$ with $\tilde{\bnu} \in \La^{\ms}$ if and only if $\bsig \in \La^{\ms}$. Similarly to the case above, $\delta_{\tilde{\bnu}}$ is formed from $\delta_{\bsig}$ by swapping the entries in the positions indexed by $i'$ and $i'+1$, and $(\delta_{\bsig})_{i'}=0$ so that $\delta_{\bsig}$ and $\delta_{\tilde{\bnu}}$ are essentially the same.
We continue in this way until we reach a bipartition $\bnu$ where there do not exist $i',j'$ as above. Then $\bnu$ satisfies the conditions of the lemma and $\bmu \in \La^{\ms}$ if and only if $\bnu \in \La^{\ms}$. 
\end{proof}

\begin{proposition} \label{P:Klesh2}
Suppose that $B$ is a core block and $\bmu \in B$. Then $\bmu\in \La^{\ms}$ if and only if $\delta_{\bmu} \in \Delta_0$.  
\end{proposition}

\begin{proof}
  Suppose that $(\mc,\bb)$ is a reduced pair for $\bmu$. By Lemma~\ref{L:KleshReduce}, we may assume that $\bb=\bz$. 
If $|B|=1$ then the block $\hat{B}$ is simple and so $\bmu \in \La^{\ms}$. So assume that $|B|>1$. 
 By Lemma~\ref{L:LoseZeros}, we may assume that there exist $i_1,i_2 \in I$ with $i_1\leq i_2$ such that
\begin{itemize}
\item For $k=1,2$, $l^{\mc}_{mk}(\bmu)=1$ for all $m < i_1$ and $l^{\mc}_{mk}(\bmu)=0$ for all $m > i_2$.
\item $(\delta_{\bmu})_m \in \{-,+\}$ for $i_1\leq m \leq i_2$.   
\end{itemize}
We prove Proposition~\ref{P:Klesh2} by induction on $|\St(\delta_{\bmu})|$. First suppose that $|\St(\delta_{\bmu})|=0$. If $\delta_{\bmu} \in \Delta_0$ then $|B|=1$, which we have assumed is not true. So suppose that $\delta_{\bmu} \notin \Delta_0$ so that there are entries equal to both $-$ and $+$ in $\delta_{\bmu}$ with the $+$ terms occuring before the $-$ terms. The abacus configuration of $\bmu$ with respect to $\mc$ looks like: 
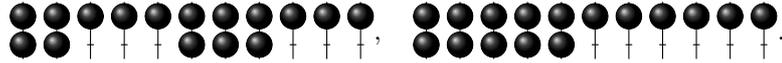
\[\abacus(bbbbbbbbbbb,bbnnnbbbnnn), \quad \abacus(bbbbbbbbbbb,bbbbbnnnnnn).\]  
Then $\bmu$ has only one removable node which is not good, hence $\bmu \notin \La^{\ms}$.

Now suppose that $|\St(\delta_{\bmu})|>0$ and Proposition~\ref{P:Klesh2} holds for any $\bsig$ in a core block with $|\St(\delta_{\bsig)}|<|\St(\delta_{\bmu})|$. Then there exist $i,i+1 \in I$ with $i<i+1$ and $(\delta_{\bmu})_i=-$ and $(\delta_{\bmu})_{i+1}=+$. By Lemma \ref{L:StillKlesh}, we can remove the removable $i+1$-node from $[\bmu]$ to obtain the Young diagram of a bipartition $\bsig$ with $\bsig \in \La^{\ms}$ if and only if $\bmu \in \La^{\ms}$. Furthermore, we have $(\delta_{\bsig})_i=(\delta_{\bsig})_{i+1}=0$ and $(\delta_{\bsig})_m=(\delta_{\bmu})_m$ for $m \neq i,i+1$ so that $\delta_{\bsig} \in \Delta_0$ if and only if $\delta_{\bmu}  \in \Delta_0$. Hence
\[\bmu \in \La^{\ms} \iff \bsig \in \La^{\ms} \iff \delta_{\bsig} \in \Delta_0 \iff \delta_{\bmu} \in \Delta_0\]
where the middle step follows from the inductive hypothesis.   
\end{proof}

In this section, we have shown that if $B$ is a core block then we can index the Specht modules corresponding to bipartitions in $B$ by a subset of the set of sequences $\delta \in \Delta$ with $n_B$ entries equal to $-$ and $p_B$ entries equal to $+$, with the simple modules indexed by sequences $\delta \in \Delta_0$. The next step is to determine the entries in the block decomposition matrix. To do this, we begin by introducing the Fock space representation of $\mathcal{U}_v(\widehat{\mathfrak{sl}}_e)$.

\subsection{The Fock space representation} \label{S:Fock}
Let $\mathcal{U}$ denote the quantized enveloping algebra $\mathcal{U}=\mathcal{U}_v(\widehat{\mathfrak{sl}}_e)$. This is a $\mathbb{Q}(v)$-algebra with generators $e_i, f_i$ for $i \in I$ and $v^h$ for $h \in P^{\vee}$; the relations may be found in~\cite{LLT}. 
Let $\F^{\ms}$ be the $\mathbb{Q}(v)$-vector space with basis $\{\Cr{\bla} \mid \bla \in \Lambda^2\}$. This becomes a $\mathcal{U}$-module under the action described in~\cite{LLT}; we call $\mathcal{F}^{\ms}$ the Fock space representation of $\mathcal{U}$. The $\mathcal{U}$-submodule $M^{\ms}$ generated by $s_{\emp^2}$ (where $\emp^2$ is the unique bipartition of $0$) is isomorphic to the irreducible highest weight module $V(\Upsilon)$ for some dominant integral weight $\Upsilon$ of $\mathcal{U}$. This module has a canonical basis (in the sense of Lusztig and Kashiwara) which is indexed by the elements of $\La^{\ms}$; we write $P^{\bmu}$ for the canonical basis element indexed by $\bmu \in \La^{\ms}$. We define the transition coefficients $d_{\bla\bmu}(v)$ so that they satisfy the equation 
\[P^{\bmu} = \sum_{\bla \in \La^2} d_{\bla\bmu}(v) \Cr{\bla};\]
then $d_{\bla\bmu}(v)=0$ unless $\bla \sim_{\ms} \bmu$. 
Ariki's Theorem and its graded analogue, below, relate the transition coefficients to the graded decomposition numbers. 

\begin{theorem} [\cite{Ariki,BK:Decomp}] \label{T:Ariki}
Suppose that $\h_n=\h_{2,n}(q,{\bf Q})$ is defined over a field of characteristic $0$. 
Suppose that $\bla,\bmu \in \La^r_n$ with $\bmu \in \La^{\ms}$. Then
\[[S^{\bla}:D^{\bmu}]_v = d_{\bla\bmu}(v).\]
\end{theorem}

The transition coefficients give us the decomposition numbers when $p=0$ and can be considered as a first approximation to the decomposition numbers when the field is arbitrary; as we shall see below, there are situations when this first approximation is correct.  In order to compute the transition coefficients, we now consider the action of $f_i \in \mathcal{U}$ on a basis element $\Cr{\bla} \in \mathcal{F}^{\ms}$. 

If $d>0$ and $\bnu,\bla \in \La^2$, write $\bnu \xrightarrow{d:i} \bla$ if $[\bla]$ is formed from $[\bnu]$ by adding $d$ nodes all of residue $i \in I$. 
If $\bnu \xrightarrow{d:i} \bla$ set
\begin{multline*}
N(\bnu,\bla) =\sum_{\mathfrak{n} \in [\bla] \setminus [\bnu]} \#\{\mathfrak{m} \text{ an addable $i$-node of } \bnu \text{ with } \mathfrak{n} \lhd \mathfrak{m}\}
-\#\{\mathfrak{m}\text{ a removable $i$-node of } \bla \text{ with } \mathfrak{n} \lhd \mathfrak{m}\}. 
\end{multline*}
For $d >0$ and $i \in I$, define $f^{(d)}_i = f^d_i\,/\,[d]! \in \mathcal{U}$, the quantum divided power of $f^d_i$. Then if $\bnu \in \La^2$, 
\[f^{(d)}_i \Cr{\bnu} = \sum_{\bnu \xrightarrow{d:i} \bla} v^{N(\bnu,\bla)} \Cr{\bla}.\] 

\begin{lemma} [{\cite[Proposition 2.3]{LyleRuff}}] \label{L:ArikiV}
Suppose $\bnu \in \La^{\ms}$ is such that if $\bsig \sim_{\ms} \bnu$ then $[S^{\bsig}:D^{\bnu}]_v = d_{\bsig\bnu}(v)$.
Let $i_1,\ldots,i_t \in I$ and $d_1,\ldots,d_t >0$. Suppose 
\[f^{(d_t)}_{i_t} \ldots f^{(d_1)}_{i_1} P^{\bnu} = \Cr{\bmu} + \sum_{\bla \ne \bmu} c_{\bla\bmu}(v) \Cr{\bla}\]
where $c_{\bla\bmu}(v) \in v \mathbb{N}[v]$ for all $\bla \ne \bmu$. 
Then $\bmu \in \La^{\ms}$ and if $\bla \neq \bmu$ then $[S^{\bla}:D^{\bmu}]_v=c_{\bla\bmu}(v)$. 
\end{lemma}

The statement of Lemma~\ref{L:ArikiV} given in~\cite{LyleRuff} is only for the case when $P^{\bnu}=\Cr{\bnu}$, but the proof for the slightly more general situation above is identical. 
Suppose $i,j \in I$. If $i \leq j$, define 
\[f_{i,j} = f_{j}  f_{j-1} \ldots f_{i+1}.\] 
If $e$ is finite and $i>j$, define 
\[f_{i,j} = f_{j}  f_{j-1}\ldots f_0 f_{e-1}\ldots f_{i+2} f_{i+1}.\] 

Suppose that $B$ is a core block and that $(\mc,\bb)$ is a reduced pair for $B$. We say that $\bb$ is flat if $e=\infty$ or if $e$ is finite and there exists $m' \in I$ and $b \in \Z$ such that $b_m=b$ if $m \geq m'$ and $b_m=b+1$ otherwise. Note that if $e$ is finite then in this case, $$m'\prec m'+1 \prec \ldots \prec e-1 \prec 0 \prec1 \prec \ldots \prec m'-1.$$ 

\begin{lemma} \label{L:ijinduce}
Suppose that $B$ is a core block and that $\re=(\mc,\bb)$ is a reduced pair for $B$ with $\bb$ flat. Let $\bnu \in B$. Suppose that $i\prec j$ with 
\[l^{\mc}_{ik}(\bnu) = b_i+1, \qquad l^{\mc}_{jk}(\bnu)=b_j, \qquad l^{\mc}_{m1}(\bnu)=l^{\mc}_{m2}(\bnu), \qquad \text{ for } k=1,2 \text{ and } i \prec m\prec j.\]
 Let $I_1=\{i\preceq m \prec j \mid l^{\mc}_{m1}(\bnu)=l^{\mc}_{m2}(\bnu)=b_m+1\}=\{i_1,i_2,\ldots,i_t\}$, where $i_1\prec \ldots \prec i_t$. Set $i_{t+1}=j$ and define 
\[f = f_{i_1,i_2} \ldots f_{i_{t-1},i_{t}} f_{i_t,i_{t+1}}\in \mathcal{U}.\]
Then \[f \Cr{\bnu} = \Cr{\bla_2} +  v \Cr{\bla_1}\]
where 
\begin{align*} l^{\mc}_{mk}(\bla_1) & =
\begin{cases}
l^{\mc}_{mk}(\bnu)-1, & m=i \text{ and } k=1, \\
l^{\mc}_{mk}(\bnu)+1, & m=j \text{ and } k=1, \\
l^{\mc}_{mk}(\bnu), & \text{otherwise}, 
\end{cases} & 
l^{\mc}_{mk}(\bla_2) & =
\begin{cases}
l^{\mc}_{mk}(\bnu)-1, & m=i \text{ and } k=2, \\
l^{\mc}_{mk}(\bnu)+1, & m=j \text{ and } k=2, \\
l^{\mc}_{mk}(\bnu), & \text{otherwise},
\end{cases}
\intertext{so that}
(\delta_{\bla_1})_m & = \begin{cases} +, & m=i, \\
-, & m=j, \\
(\delta_{\bnu})_m, & \text{otherwise},
\end{cases} & 
(\delta_{\bla_2})_m & = \begin{cases} -, & m=i, \\
+, & m=j, \\
(\delta_{\bnu})_m, & \text{otherwise}.
\end{cases}
\end{align*}
\end{lemma}

\begin{proof}
For $1 \leq s \leq t$ and $k=1,2$, define $\bsig_k^s$ by 
\[l^{\mc}_{mk'}(\bsig_k^s)=
\begin{cases}
l^{\mc}_{mk'}(\bnu)-1, & m=i_s \text{ and } k=k', \\
l^{\mc}_{mk'}(\bnu)+1, & m=j \text{ and } k=k', \\
l^{\mc}_{mk'}(\bnu), & \text{otherwise},
\end{cases}\]
so that $\bla_k = \bsig_k^1$ for $k=1,2$. We claim that if $1 \leq s \leq t$ then 
\[f_{i_{s},i_{s+1}} \ldots f_{i_t,i_{t+1}}\Cr{\bnu}  = \Cr{\bsig_2^s} + v \Cr{\bsig_1^s}.\]
Since $\bb$ is flat, $[\bnu]$ has exactly two addable $(i_t+1)$-nodes, one in the first component and one in the second, and no removable $(i_t+1)$-nodes. If we add an $(i_t+1)$-node to component $k$, for $k=1,2$, then the subsequent bipartition has exactly one addable $(i_t+2)$-node, in component $k$, and no removable $(i_t+2)$-nodes. Continuing in this way, we see that 
\[f_{i_t,i_{t+1}} \Cr{\bnu} = \Cr{\bsig^t_2} + v \Cr{\bsig^t_1}.\]
Now suppose $1 \leq s \leq t-1$ and consider $f_{i_s,i_{s+1}} \Cr{\bsig^{s+1}_k}$, where $k = 1,2$. 
If $i_s+1=i_{s+1}$ then $[\bsig^{s+1}_k]$ has exactly one addable $i_{s+1}$-node, which is in component $k$, and no removable $i_{s+1}$-nodes, and so $f_{i_{s+1}} \Cr{\bsig^{s+1}_k} = \Cr{\bsig^{s}_k}$. Otherwise, as above, 
\[f_{i_s,i_{s+1}-1} \Cr{\bsig^{s+1}_k} = \Cr{\btau_2} + v \Cr{\btau_1},\]
where $[\btau_{k'}]$ is obtained by successively adding nodes to component $k'$ of $[\bsig^{s+1}_k]$, for $k'=1,2$. However, if $k' \neq k$, $[\btau_{k'}]$ does not have any addable $i_{s+1}$-nodes so $f_{i_{s+1}}\Cr{\btau_{k'}}=0$. If $k=k'=1$ (resp. $k=k'=2$) then $[\btau_{k}]$ has an addable $i_{s+1}$-node on component $1$ (resp. on component $2$) and a removable $i_{s+1}$-node on component $2$ (resp. on component 1) and so $f_{i_{s+1}} \Cr{\btau_{k'}} = v^{-1} \Cr{\bsig_{k'}}$ (resp. $f_{i_{s+1}}\Cr{\btau_{k'}} = \Cr{\bsig_{k'}}$). So in both cases we have 
\[f_{i_s,i_{s+1}} \Cr{\bsig^{s+1}_k} = \Cr{\bsig^{s}_k}.\]
The lemma follows. 
\end{proof}

\subsection{Decomposition matrices} \label{S:Decomp}

\begin{theorem} \label{T:Flat}
Let $B$ be a core block and suppose that $\re=(\mc,\bb)$ is a reduced pair for $B$ with $\bb$ flat. 
Let $\bla,\bmu \in B$ with $\bmu \in \La^{\mc}$. Then
\[ [S^{\bla}:D^{\bmu}]_v = d_{\bla\bmu}(v) = \begin{cases} 
v^{\ell(\delta_{\bla},\delta_{\bmu})}, & \delta_{\bla} \ra \; \delta_{\bmu}, \\
0, & \text{otherwise}.
\end{cases}\]
\end{theorem}

\begin{proof}
We work by induction on $|\St(\bmu)|=m_B$. If $|\St(\bmu)|=0$ then $B$ is simple and the result holds. So suppose $|\St(\bmu)|>0$ and the theorem holds for all core blocks $B'$ with $m_{B'}<m_B$. Choose $(i,j) \in \St(\bmu)$ such that there does not exist $(i',j') \in \St(\bmu)$ with $i<i'<j'<j$. Let $\bnu$ be the bipartition defined by
\[l^{\mc}_{mk}(\bnu) = \begin{cases}
l^{\mc}_{mk}(\bmu) +1, & m=i \text{ and } k=2, \\
l^{\mc}_{mk}(\bmu)-1, & m=j \text{ and } k=2, \\
l^{\mc}_{mk}(\bmu), & \text{otherwise}.
\end{cases} \]
Then $\bnu \in B'$ for some core block $B'$ such that $(\mc,\bb)$ is a reduced pair for $B'$. Also $\St(\bnu) = \St(\bmu) \setminus \{(i,j)\}$ so $\bnu \in \La^{\ms}$. Hence $\bnu$ satisfies the conditions of the inductive hypothesis, so by Theorem~\ref{T:Ariki}, 
\[P^{\bnu} = \sum_{\delta_{\bsig} \ra \delta_{\bnu}} v^{\ell(\bsig,\bnu)} \Cr{\bsig}.\]
The bipartition $\bnu$ with reduced pair $(\mc,\bb)$, together with the pair $(i,j)$ described above, satisfies the conditions of Lemma~\ref{L:ijinduce}. 
We define $f$ as in that lemma and compute each term $f \Cr{\bsig}$ that appears in the sum above, so that 
\[f P^{\bnu} = \sum_{\delta_{\bla} \ra \delta_{\bmu}} v^{\ell(\bla,\bmu)} \Cr{\bla}.\]
Since $\ell(\bla,\bmu)>0$ for all $\bla \neq \bmu$ which appear in the sum, we have satisfied the conditions of Lemma~\ref{L:ArikiV} and so the theorem holds. 
\end{proof}

Since there is always a reduced pair $(\ms,\bb)$ with $\bb$ flat when $e=\infty$ or $n<e$, this describes the block decomposition matrices for core blocks in these cases, recovering previous results; we note in particular the similarity between Theorem~\ref{T:Flat} and~\cite[Theorem~3]{LM:FockI}. For the rest of this section, we assume that $e$ is finite. 
Suppose that $B$ and $B'$ are any two core blocks with $n_B=n_{B'}$ and $p_{B}=p_{B'}$. Then we have a bijection $\Phi:B \rightarrow B'$ such that $\delta_{\bla}$ and $\delta_{\Phi(\bla)}$ are essentially the same for all $\bla \in B$. If $\bmu \in B$ then by Lemma~\ref{P:Klesh2}, we have that $\bmu \in \La^{\ms}$ if and only if $\Phi(\bmu) \in \La^{\ms}$. 

Given two core blocks $B$ and $B'$, we say that $B$ and $B'$ are $\delta$-equivalent if $n_B=n_{B'}$ and $p_B=p_{B'}$ and if the map $\Phi$ defined above preserves the block decomposition matrix of $B$, that is, if $\bla,\bmu \in B$ with $\bmu \in \La^{\ms}$ then $[S^{\bla}:D^{\bmu}]_v = [S^{\Phi(\bla)}:D^{\Phi(\bmu)}]_v$. In fact, we will show that any two core blocks $B$ and $B'$ with $n_B=n_{B'}$ and $p_B=p_{B'}$ are $\delta$-equivalent. Theorem~\ref{T:Flat} shows that this also holds when $e=\infty$. 

\begin{lemma} \label{L:Scopesy}
Let $B$ be a core block and suppose that $\re=(\mc,\bb)$ is a reduced pair for $B$.
Suppose that $i\in I$ and 
\[b_{i+1} > b_{i}+x_i \qquad \text{ where } x_i = \begin{cases} 1, & i=e-1, \\ 0, & \text{otherwise}.\end{cases}\]
Define a map $\Phi_i:B \rightarrow \La^2$ which sends $\bla \in B$ to the bipartition whose Young diagram is obtained by removing all possible $(i+1)$-nodes from $[\bla]$. Then $\Phi_i$ gives a bijection between $B$ and a core block $B'$ such that if $\bla\in B$ then $\delta_{\bla} = \delta_{\Phi_i(\bla)}$. 
\end{lemma}

\begin{proof}
Suppose $\bla\in B$. From our assumptions on $\bb$, $[\bla]$ has no addable $(i+1)$-nodes. If $\bnu \in B$ then by Lemma~\ref{L:GenBlock}, $l^{\ms}_{m1}(\bla)+l^{\ms}_{m2}(\bla) = l^{\ms}_{m1}(\bnu)+l^{\ms}_{m2}(\bnu)$ for all $m \in I$ so that the number of removable $(i+1)$-nodes on both $[\bla]$ and $[\bnu]$ is 
\[l^{\ms}_{i+11}(\bla)-l^{\ms}_{i1}(\bla)-x_i+l^{\ms}_{i+12}(\bla)-l^{\ms}_{i2}(\bla)-x_{i} = l^{\ms}_{i+11}(\bnu)-l^{\ms}_{i1}(\bnu)-x_i+l^{\ms}_{i+12}(\bnu)-l^{\ms}_{i2}(\bnu)-x_i.\]
Hence $\Phi_i(\bla)$ and $\Phi_i(\bnu)$ lie in the same block, $B'$ say. If $i \neq e-1$, $\Phi_i$ acts on the abacus configurations by swapping the runners $i$ and $i+1$ on each component, so that $\delta_{\bla} = \delta_{\Phi_i(\bla)}$ and so by Corollary~\ref{C:GenBlock}, $\Phi_i$ is a bijection between $B$ and $B'$. 
If $i=e-1$, then $(\mc,\bb')$ is a reduced pair for $B'$, where
\[b'_m = \begin{cases} b_{e-1}+1, & m=0, \\
  b_0 -1, & m=e-1,\\
  b_m, & \text{otherwise}.\end{cases}\]
It is straightforward to check that if $m \neq 0,e-1$ then
\[b_m \prec b_{e-1} \iff b_m' \prec b_0', \qquad
b_{e-1} \prec b_m \prec b_0 \iff b'_0 \prec b'_m \prec b'_{e-1}, \qquad
b_0 \prec b_m \iff b'_{e-1} \prec b'_j,\]
so that again $\delta_{\bla} = \delta_{\Phi_i(\bla)}$ and $\Phi_i$ is a bijection between $B$ and $B'$. 
\end{proof}

\begin{lemma} \label{L:Sc2}
  Suppose that we have the conditions of Lemma~\ref{L:Scopesy}. Keep the notation of that lemma so that $B'=\Phi_i(B)$.
Suppose $B'$ is such that if $\bsig,\btau \in B'$ with $\btau \in \La^{\ms}$ then
$[S^{\bsig}:D^{\btau}]_v = d_{\bsig\btau}(v)$. 
  Then $B'$ and $B$ are $\delta$-equivalent. 
\end{lemma}

\begin{proof}
Suppose that if $\bla \in B$ then $\bla$ has $t$ removable $(i+1)$-nodes; equivalently if $\bsig \in B'$ then $\bsig$ has $t$ addable $(i+1)$-nodes and no removable $(i+1)$-nodes. Let $\Psi_i = \Phi_i^{-1}$.  
If $\bsig \in B'$ then
\[f^{(t)}_{i+1} \Cr{\bsig} = \Cr{\Psi_i(\bsig)}.\]
Let $\btau \in B'$ be such that $\btau \in \La^{\ms}$. Then
\[f^{(t)}_{i+1} P^{\btau} = f^{(t)}_{i+1} \sum_{\bsig} d_{\bsig\btau}(v) \Cr{\bsig} = \sum_{\bsig} d_{\bsig\btau}(v)\Cr{\Psi_i(\bsig)}.\]
Hence by repeated application of Lemma~\ref{L:ArikiV}, $[S^{\Psi_i(\bsig)}:D^{\Psi_i(\btau)}]_v = d_{\bsig\btau}(v) =[S^{\bsig}:D^{\btau}]_v$ for $\bsig \in B'$. Since $\delta_{\bsig}=\delta_{\Psi_i(\bsig)}$ for all $\bsig \in B'$, the lemma follows. 
  \end{proof}

\begin{lemma} \label{L:Flat}
Suppose that $B$ is a core block with reduced pair $(\mc,\bb)$. Then there exists a core block $B'$ with reduced pair $(\mc,\bb')$ such that $B$ and $B'$ are $\delta$-equivalent and  $\bb'$ is flat. 
\end{lemma}

\begin{proof}
Following Lemma~\ref{L:Scopesy}, if there exists $i \in I \,\setminus \,\{e-1\}$ with $b_{i}<b_{i+1}$ then we may apply the map $\Phi_i$ to $B$. If $b_0-b_{e-1}>1$ then we may apply the map $\Phi_{e-1}$ to $B$. By repeatedly applying these maps, we obtain a core block $B' = \Phi_{i_1} \Phi_{i_2}\ldots \Phi_{i_z}B$ where $(\mc,\bb')$ is a reduced pair for $B'$ and $\bb'$ is flat. By Theorem~\ref{T:Flat}, $[S^{\bsig}:D^{\btau}]_v = d_{\bsig\btau}(v)$ for all $\bsig,\btau \in B'$ with $\btau \in \La^{\mc}$. Hence by Lemma~\ref{L:Sc2}, we have that $B$ and $B'$ are $\delta$-equivalent.  
\end{proof}

\begin{corollary}
Suppose that $B$ and $B'$ are core blocks. 
If $n_B=n_{B'}$ and $p_B=p_{B'}$ then $B$ and $B'$ are $\delta$-equivalent. 
\end{corollary}

\begin{theorem} \label{MainAll}
Let $B$ be a core block and suppose $\bla,\bmu \in B$ with $\bmu \in \La^{\ms}$. Then
\[ [S^{\bla}:D^{\bmu}]_v = \begin{cases} 
v^{\ell(\delta_{\bla},\delta_{\bmu})}, & \delta_{\bla} \ra \; \delta_{\bmu}, \\
0, & \text{otherwise}.
\end{cases}\]
\end{theorem}

\begin{proof}
Following Lemma~\ref{L:Flat}, we know that $B$ is $\delta$-equivalent to a core block $B'$ with reduced pair $(\mc,\bb')$ where $\bb'$ is flat. The theorem then follows from Theorem~\ref{T:Flat}. 
\end{proof}

\subsection{Summary and examples} \label{S:Wrapup}
We collect together the main results of this paper.

\begin{theorem} \label{T:Summary}
Let $B$ be a core block and let $B_0 = B \cap \La^{\mc}$. Let $n=n_B, \, p=p_B$ and $m=\min\{n,p\}$.
\begin{enumerate}
\item We can identify a bipartition $\bla \in B$ with a sequence $\delta_{\bla} \in \Delta$ which contains $n$ entries equal to $-$ and $p$ entries equal to $+$. If $\bnu \in B$ then $\delta_{\bnu}$ is formed by permuting these $n+p$ entries. Hence 
\[|B|=\binom{n+p}{m}.\]
\item Suppose $\bmu \in B$. Then $\bmu \in B_0$ if and only if $\delta_{\bmu} \in \Delta_0$. Hence \[|B_0| = \binom{n+p}{m} - \binom{n+p}{m-1}.\]
\item We have $\wt(B)=m$.  
\item There exists $d \in \Z$ such that if $\ms=(s_1,s_2)$ then if $e$ is finite (resp. $e=\infty$) then $s_1 \equiv n+d \mod e$ and $s_2 \equiv p+d \mod e$ (resp. $s_1=n+d$ and $s_2 = p+d$). 
\item Suppose that $\bla,\bmu \in B$ with $\bmu \in B_0$. Then 
\[ [S^{\bla}:D^{\bmu}]_v = \begin{cases} 
v^{\ell(\delta_{\bla},\delta_{\bmu})}, & \delta_{\bla} \ra \; \delta_{\bmu}, \\
0, & \text{otherwise}.
\end{cases}\]
\end{enumerate}
\end{theorem}

\begin{proof}
All these results appear in the previous section, apart from the formula for $|B_0|$. To see that this holds, first suppose that $n \geq p$. Identify each sequence with $n$ entries equal to $-$ and $p$ entries equal to $+$ with a $(n,p)$-tableau by putting the entries corresponding to $-$ in the top row and to $+$ in the bottom row. The tableau is standard if and only if the sequence lies in $\Delta_0$. Since the number of standard $(n,p)$-tableaux is given by the formula above, we are done. The case that $n<p$ follows by symmetry. 
\end{proof}

For $\bla,\bnu$ in a core block $B$, let $\dim(\Hom(S^{\bla},S^{\bnu}))_v$ denote the graded dimension of the homomorphism space between $S^{\bla}$ and $S^{\bnu}$. The next result appears in George Witty's 2020 thesis~\cite{WittyThesis}. Although Witty only proves the case where $(\mc,\bz)$ is a reduced pair for $B$, he conjectures that the result holds for an arbitrary reduced pair~\cite[Conjecture~4.32]{WittyThesis}. It is possible that one could prove this conjecture using Scopes equivalence~\cite{Webster}. 

\begin{theorem}[{\cite[Theorem~4.27]{WittyThesis}}] \label{T:Witty}
Suppose that $B$ is a core block and that $(\mc,\bz)$ is a reduced pair for $B$. 
Suppose that $e\neq 2$ or that $n_B \neq p_B$. 
Let $\bla,\bnu \in B$. Then
\[ \dim(\Hom(S^{\bla},S^{\bnu}))_v = \begin{cases} 
v^{\ell(\delta_{\bla},\delta_{\bnu})}, & \delta_{\bla} \rra \; \delta_{\bnu}, \\
0, & \text{otherwise}.
\end{cases}\]
If $e=2$ or $n_B=p_B$ then the formula above gives a lower bound for $\dim(\Hom(S^{\bla},S^{\bnu}))_v$. 
\end{theorem}

Witty's theorem is an application of results proved in his thesis about homomorphisms between Specht modules. His proofs involve very intricate calculations within the cyclotomic KLR algebras.

\medskip
It is likely that a formula similar to that of Theorem~\ref{T:Summary} (5) holds for the graded decomposition numbers for the type $B$ Schur algebras when $\bmu \notin B_0$ and it seems to us that an argument via induction on $|\mathcal{S}(\delta_{\bmu})|$, using the techniques developed in this paper, would prove the result. We thank the referee for asking this question, since our assumption when writing the paper was that our methods would not work for the Schur algebras.  

We end this paper with some examples. 

\begin{ex}
Suppose that $e=5$ and $\ms=(3,1)$. Let $\bla = ((7,5,4,3^2,2,1^2),(14,10,6,4,3,2^2,1))$ and let $B$ be the core block containing $\bla$. Then $((18,16),(2,4,3,1,5))$ is a reduced pair for $B$. By drawing the abacus configurations, we can see that $n_B=3$ and $p_B=1$. 
The four bipartitions in the block are given by the abacus configurations below. 

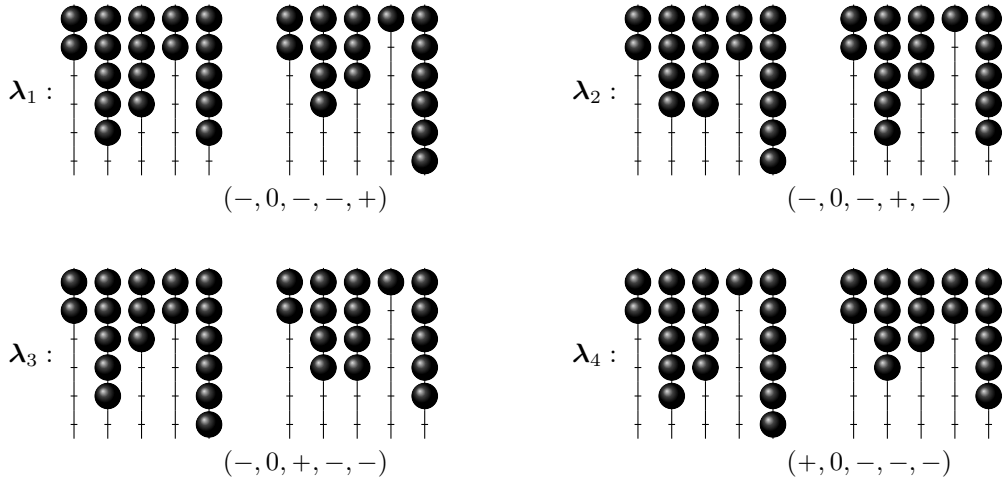
\begin{align*}
\bla_1: \abacus(bbbbb,bbbbb,nbbnb,nbbnb,nbnnb,nnnnn) & \qquad  \abacus(bbbbb,bbbnb,nbbnb,nbnnb,nnnnb,nnnnb) & 
\bla_2: \abacus(bbbbb,bbbbb,nbbnb,nbbnb,nnnnb,nnnnb) & \qquad \abacus(bbbbb,bbbnb,nbbnb,nbnnb,nbnnb,nnnnn) \\
& (-,0,-,-,+) && (-,0,-,+,-) \\ \\
\bla_3: \abacus(bbbbb,bbbbb,nbbnb,nbnnb,nbnnb,nnnnb) & \qquad  \abacus(bbbbb,bbbnb,nbbnb,nbbnb,nnnnb,nnnnn) &
\bla_4: \abacus(bbbbb,bbbnb,nbbnb,nbbnb,nbnnb,nnnnb) & \qquad \abacus(bbbbb,bbbbb,nbbnb,nbnnb,nnnnb,nnnnn) \\
& (-,0,+,-,-) && (+,0,-,-,-)
\end{align*}

Then the decomposition matrix for $B$ is 
\[\begin{array}{|c|c|ccc|} \hline
\bla_1 &(-,0,-,-,+) & 1& & \\
\bla_2 &(-,0,-,+,-) & v & 1  &   \\ 
\bla_3 &(-,0,+,-,-) & & v & 1  \\
\bla_4 &(+,0,-,-,-) & &  & v \\ \hline 
\end{array}\]

\end{ex}

\begin{ex}
Suppose that $B$ is a core block with $n_B=2$ and $p_B=3$. Then $|B|=10$ and by removing the $0$s from the sequences $\{\delta_{\bla} \mid \bla \in B\}$ we get Specht modules indexed by the $10$ elements below. The block decomposition matrix for $B$ is equal to

\[\begin{array}{|c|ccccc|} \hline
(-,-,+,+,+) & 1& & & &\\
(-,+,-,+,+) & v & 1 & & &   \\ 
(+,-,-,+,+) & & v & 1 & &\\
(-,+,+,-,+) & & v & & 1& \\
(+,-,+,-,+) & v & v^2 & v & v& 1 \\
(+,+,-,-,+) & v^2 &  & & & v\\
(-,+,+,+,-) & & & &v &  \\
(+,-,+,+,-) & & & & v^2& v\\
(+,+,-,+,-) & & & v & &v^2\\
(+,+,+,-,-) & & & v^2 && \\ \hline 
\end{array}\]

\smallskip
Note that this is the matrix that appears in~\cite[Example~6]{LM:FockI}. 
\end{ex}

\end{document}